\documentclass[a4paper]{amsart}

\usepackage{amsaddr}
\author{Ben Andrews}
\address{School of Mathematics and Statistics, University of Sheffield, Sheffield S3 7RH}
\email{benand34@gmail.com}
\author{Jonathan Jordan}
\address{School of Mathematics and Statistics, University of Sheffield, Sheffield S3 7RH}
\email{jonathan.jordan@sheffield.ac.uk}

\title{Fragility of non-convergence in preferential attachment graphs with three types}

\usepackage{graphicx}
\usepackage{amsthm}
\usepackage[margin=1in]{geometry}

\newtheorem{theorem}{Theorem}[section]

\newtheorem{proposition}[theorem]{Proposition}


\begin{document}

\maketitle

\begin{abstract}
Preferential attachment networks are a type of random network where new nodes are connected to existing ones at random, and are more likely to connect to those that already have many connections.  We investigate further a family of models introduced by Antunovi\'{c}, Mossel and R\'{a}cz where each vertex in a preferential attachment graph is assigned a type, based on the types of its neighbours. Instances of this type of process where the proportions of each type present do not converge over time seem to be rare.

Previous work found that a ``rock-paper-scissors" setup where each new node's type was determined by a rock-paper-scissors contest between its two neighbours does not converge. Here, two cases similar to that are considered, one which is like the above but with an arbitrarily small chance of picking a random type and one where there are four neighbours which perform a knockout tournament to determine the new type.

These two new setups, despite seeming very similar to the rock-paper-scissors model, do in fact converge, perhaps surprisingly.
\end{abstract}

\section{Introduction}
In this paper, we consider a model for randomly growing networks that have nodes of different types, where the types of nodes are chosen based on what they see connected to them when they join the network. These types could represent, for example, brand preferences, where people choose their preference based on those of their friends or those of celebrities.

The model we consider is based on the linear preferential attachment graph, where nodes are more likely to connect to those that already have a lot of connections, similar to the influence of celebrities in the example above. The type assignment model on preferential attachment graphs was introduced by Antunovi\'c, Mossel and R\'acz in \cite{amr}; the general set-up provides for $N$ types and a flexible family of type assignment rules based on the types of neighbouring vertices. They proved a strong result for the case with two types, that the proportion of each type present over time almost surely converges to a limit, which is a fixed point of a one-dimensional differential equation and, depending on the choice of type assignment mechanism, may be random.

They also conjecture (Conjecture 3.2 of \cite{amr}) that a similar result is true for three or more types. However, in previous work \cite{hj}, one of the co-authors showed that this is not true for a “rock-paper-scissors” case, where each node is connected to two others and the type of the new node is chosen by the winner of a rock-paper-scissors contest between its two neighbours. It seems, with three types at least, that these exceptions are unusual and special, and most natural cases do converge almost surely.

Here, we consider some variations on the rock-paper-scissors model, mainly on one which is very similar, but with a small probability $h$ of taking a random type, rather than performing the rock-paper-scissors process (this can be considered to be a very small perturbation of the rock-paper-scissors model). Indeed, this model does converge almost surely to one third of the nodes present being each type. We will also consider a model where new nodes receive four neighbours and these four types perform a “knockout tournament” to decide the type of the new node. The equivalent case with $m=2$ is the original rock-paper-scissors case, but in the $m=4$ case, this model also converges almost surely.

\section{The Antunovi\'{c}-Mossel-R\'acz framework}\label{framework}
The framework introduced by Antunovi\'{c}, Mossel and R\'acz in \cite{amr} considers a standard
preferential attachment graph where the new vertex connects to $m$ existing vertices.  Preferential attachment as a network model was popularised by Barab\'{s}i and Albert \cite{scalefree1999}, and a rigorous mathematical formulation followed in \cite{brst}.  The specific version of preferential attachment used in \cite{amr} and in the present paper is the ``independent model'' of \cite{bbcs}.  The initial graph is called $G_0$, and then for every $t \in \mathbb{N}$ a new vertex is connected to $m$ vertices in $G_{t-1}$ (allowing multiple edges) where the probability of being connected to each other vertex is proportional to its degree, and the $m$ vertices are chosen independently; this gives $G_t$.

For the framework of \cite{amr}, each vertex is one of $N$ types (types notated ${1, …, N}$) and a vertex receives a type when it joins the network; this type never changes. The type of a new vertex is determined by the types of all its neighbours.  To define the type assignment rule, for each vector $\mathbf{u}$ of length $N$ with elements summing to $m$, we define a vector $\mathbf{p_u}$, also of length $n$ and giving a probability distribution on $\{1, …, N\}$.  If the number of each type in the new vertex’s neighbours is given by $\mathbf{u}$ then the probabilities of each type for the new vertex are given by $\mathbf{p_u}$.  We will generally assume that each type is present in the initial graph $G_0$, though this is not necessary in all examples.

A simple example is where $\mathbf{p_u} = \mathbf{u}/m$; this is known as the linear model, and has special properties.  For more general models, in \cite{amr}, Antunovi\'{c}, Mossel and R\'acz demonstrate that the sequence of vectors $\mathbf{x}_n$ which give the proportions of degrees of each type is a stochastic approximation process, meaning that we can write
$$\mathbf{x}_{n+1}-\mathbf{x}_n=\frac{1}{n}(P\left(\mathbf{x}_n)+\xi_{n+1}+R_n\right).$$  Here $P$ is an $N-1$-dimensional vector field $P$ which depends on the $\mathbf{p_u}$, and, letting $(\mathcal{F}_n)_{n\in\mathbb{N}}$ be the natural filtration of the process, $\mathbb{E}(\xi_{n+1}|\mathcal{F}_n)=0$ and $R_n$ is $\mathcal{F}_n$-measurable and satisfies $R_n\to 0$ and $\sum_{n=1}^{\infty}\frac{|R_n|}{n}$ is finite almost surely.  This means that we can apply standard results on stochastic approximation, as given for example in Pemantle \cite{Pemantle}, and to do this analysis of the vector field $P$ is key to understanding the behaviour of these models.  When $N=2$, a full analysis is given in \cite{amr}, showing that the proportions of each type converge to a stationary point of $P$, but when $N>2$ it is hard to give a general analysis of $P$ due to the variety of behaviour of higher dimensional dynamical systems; the relationship between stochastic approximation and dynamical systems, giving an idea of the complications which can arise, is covered in detail in Bena\"{i}m \cite{benaim}.

In \cite{hj}, Haslegrave and Jordan considered a type assignment system with $N=3$ types, labelled “rock”, “paper” and “scissors”, and $m=2$.  The type of a new node is determined by a rock-paper-scissors competition between the types of its two neighbours, so that the winner becomes the type of the new node. If both neighbours are the same, the new node takes their type.  In the notation above, we have $$p_{(2,0,0)}=(1,0,0), p_{(0,2,0)}=(0,1,0),p_{(0,0,2)}=(0,0,1);$$
$$p_{(1,1,0)}=(0,1,0),p_{(1,0,1)}=(1,0,0), p_{(0,1,1)}=(0,0,1).$$  The results of \cite{hj} showed that in this model the proportions of the types did not converge and instead cycled.

\section{Small perturbation case}\label{perturb_sect}
In this section, we consider a small perturbation case of the rock-paper-scissors model described above.  In this perturbation case, there is a small probability $h$ (any $h<1$ can be used) of ignoring the result of the above process, and the new node just taking a new type at random, and thus a probability $1 - h$ of the new type being selected using the original rock-paper-scissors method.

In this way, for a small $h$, the process can be very close to that of the original rock-paper-scissors model, and the perturbation can be arbitrarily small as $h$ gets very close to zero. Note that the case with $h=0$ is the original model of \cite{hj}.

Define $k = h/3$. In the notation from section \ref{framework}, for this model we have

$$p_{(1,1,0)} = (k, 1-2k, k), p(0,1,1) = (k, k, 1-2k),$$
$$p_{(1,0,1)} = (1-2k, k, k), p(2,0,0) = (1-2k, k, k),$$
$$p_{(0,2,0)} = (k, 1-2k, k), p(0,0,2) = (k, k, 1-2k).$$
Let $A_n$, $B_n$ and $C_n$ denote the normalised proportions of types 1, 2 and 3 respectively in $G_n$.

Simulation results suggest that for this model, it may be that the proportions of each type do not behave as they do in the original model, but instead may converge to $(1/3, 1/3, 1/3)$. Figure \ref{fig:1} shows the results of a simulation with $h=0.05$.

\begin{figure}
  \includegraphics[scale=0.7]{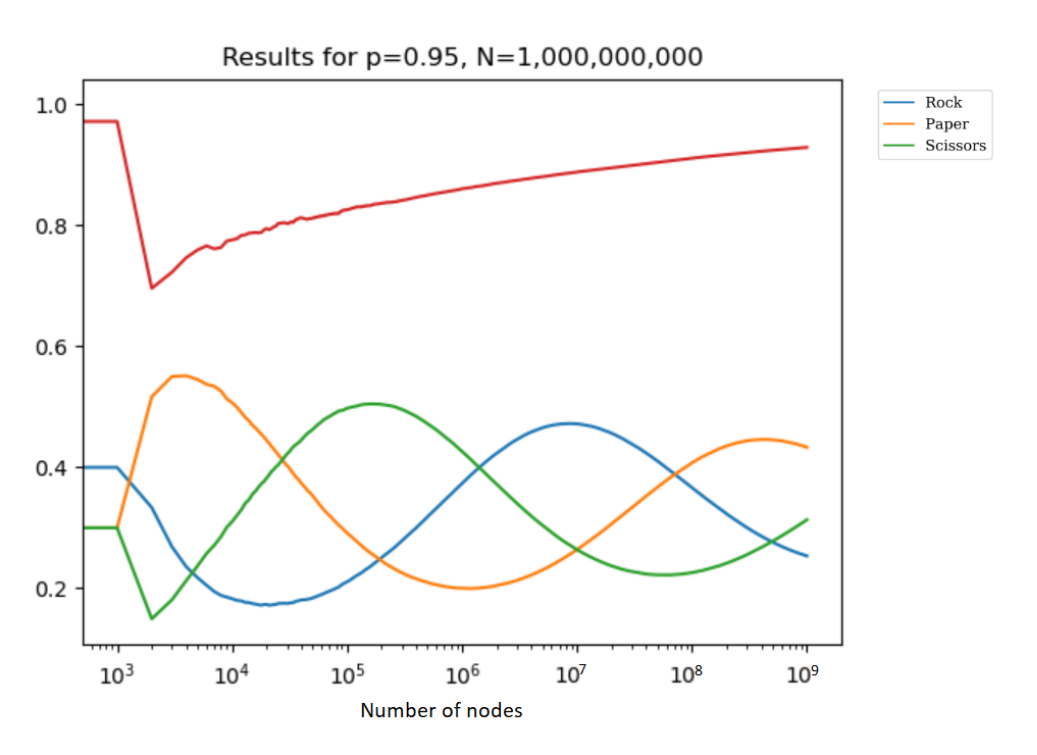}
  \caption{Results for a simulation of the perturbation model with $h=0.05$. Here, $G_0$ = a complete graph with three of each type. The red line shows the value of the product $27X_nY_nZ_n$. }
  \label{fig:1}
\end{figure}

From Figure \ref{fig:1}, and comparing with the original model \cite{hj} it seems that this model likely converges. Let $X_n, Y_n, Z_n$ denote the (normalised) sums of degrees of the nodes of types 1, 2 and 3 respectively in $G_n$, and define the product $M_n=X_nY_nZ_n$; from the figure, this appears to be increasing and converging to $1$, where in the original model it converges to a random limit. The oscillations of the proportions are also getting smaller each time here, which suggests they may eventually all converge to $1/3$.

This motivates the main result of this section:

\begin{theorem}\label{perturb_thm}
For the perturbation model with any $h > 0$, $(A_n, B_n, C_n)$ and $(X_n, Y_n, Z_n)$ converge almost surely to $(1/3, 1/3, 1/3)$.
\end{theorem}

This will follow from Proposition \ref{perturb_prop} later.

For this model, the vector field $P$, defined by (3.1) of \cite{amr}, on the triangle $\Delta^2$ is given by the components
$$P_1(x, y, z) = \frac{x}{2} (z-y)+y(x+z)k-x(x+2z)k+\frac{1}{2} (y^2+z^2 )k,$$
$$P_2(x, y, z) = \frac{y}{2} (x-z)+z(x+y)k-y(2x+y)k+\frac{1}{2} (x^2+z^2 )k,$$
$$P_3(x, y, z) = \frac{z}{2} (y-x)+x(y+z)k-z(2y+z)k+\frac{1}{2} (x^2+y^2 )k.$$

The following result tells us that $\Lambda(x,y,z)=-xyz$ is a Lyapunov function for this vector field.  Because $(X_n, Y_n, Z_n)$ is a stochastic approximation process, standard results on stochastic approximation with a Lyapunov function (for example in Pemantle \cite{Pemantle}) will allow us to use it to conclude Theorem \ref{perturb_prop}.

\begin{proposition}\label{perturb_prop}
The product $xyz$ is constant on the trajectories of $P$ only when $(x,y,z) = (1/3,1/3,1/3)$.
Otherwise, it is strictly increasing on said trajectories.
\end{proposition}

\begin{proof}
We have $\frac{d(xyz)}{dt} = xyP_3 + xzP_2 + yzP_1$. Substituting in the components above gives:

$$\frac{d(xyz)}{dt} = \frac{xyz}{2}(y-x) + {x^2}y(y+z)k - xyz(2y+z)k + \frac{1}{2}xy({x^2}+{y^2})k $$
$$+\frac{xyz}{2}(x-z) + x{z^2}(x+y)k - xyz(2x+y)k + \frac{1}{2}xz({x^2}+{z^2})k $$
$$+\frac{xyz}{2}(z-y) + {y^2}z(x+z)k - xyz(x+2z)k + \frac{1}{2}yz({y^2}+{z^2})k.$$

This reduces to
$$\frac{d(xyz)}{dt} = \frac{k}{2}(x^2(1-x) + y^2(1-y) + z^2(1-z) - 6xyz).$$

Indeed, at $(x,y,z) = (1/3,1/3,1/3)$, $\frac{d(xyz)}{dt} = \frac{1}{3}*\frac{2}{3} - \frac{6}{27} = 0$, and when $(x,y,z) = (1,0,0)$, $(0,1,0)$ or $(0,0,1)$, $\frac{d(xyz)}{dt} = 0$.

Now, using that $z = 1 - x - y$, we can write
$$\frac{d(xyz)}{dt} \propto x^2 - x^3 + y^2 - y^3 + (1-x-y)^2 - (1-x-y)^3 - 6xy(1-x-y)$$
which reduces to
$$\frac{d(xyz)}{dt} \propto x - x^2 + y - y^2 - 10xy + 9x^2y + 9xy^2.$$

Define $f = x - x^2 + y - y^2 - 10xy + 9x^2y + 9xy^2$.  We can classify the stationary points of $f$, and since it is the derivative multiplied by a constant, it will retain the signs of the derivative (and all the behaviour regarding being positive, negative or zero).

Its partial derivative with respect to $x$ is $\frac{\partial f}{\partial x} = (9y - 1)(2x+y-1)$. Therefore, at all stationary points, either $y = 1/9$ or $2x+y=1$. 

Similarly, $\frac{\partial f}{\partial y} = (9x - 1)(x+2y-1)$ and so at all stationary points, either $x = 1/9$ or $x+2y=1$. From this we get all the stationary points of $f$ in the form $(x,y,z)$: they are $(1/3,1/3,1/3)$, $(1/9,1/9,7/9)$, $(1/9,7/9,1/9)$ and $(7/9,1/9,1/9)$.

Now, these stationary points are to be classified. We calculate the second partial derivatives as $\frac{\partial^2f}{\partial x^2} = 18y - 2$, $\frac{\partial^2f}{\partial x \partial y} = 18(x+y) - 10$ and $\frac{\partial^2f}{\partial y^2} = 18x - 2$ and define these as $A$, $B$ and $C$ respectively. Then, for the stationary point $(1/3,1/3,1/3)$, the values of the derivatives are $A=4$, $B=2$ and $C=4$, and so $A>0$ and $AC > B^2$. This means the point is a local minimum. For all other stationary points, the value of $AC$ is zero and of $B^2$ is $36$, and so $B^2 > AC$. This means those points are saddle points.

The only way that the local minimum at $(1/3,1/3,1/3)$ could not be a global minimum on the simplex $\Delta^2$ is if the value of the function is lower than zero on the boundary of $\Delta^2$, as here the point may not be a local minimum due to the behaviour outside of the simplex. However, on the boundary, at least one of $x$, $y$ and $z$ are zero. And thus, $\frac{d(xyz)}{dt}$ is non-negative, because $-6xyz = 0$ and all other parts of the function are never negative for $x,y,z \leq 1$.  In fact, $\frac{d(xyz)}{dt}$ is negative on the boundary except at the corner points, and at these points inspection of the vector field shows that trajectories started there are also strictly decreasing.

In conclusion, the only minimum point is $(1/3,1/3,1/3)$, at which the value of $\frac{d(xyz)}{dt}$ is zero. It is positive everywhere else in the interior, since there are no other minimum points, and $xyz$ is also decreasing on trajectories started on the boundary.  In other words, the product $xyz$ is increasing on the trajectories of $P$, except at $(1/3,1/3,1/3)$ where it is constant, as required.
\end{proof}

\begin{proof}[Proof of Theorem \ref{perturb_thm}.]
Proposition \ref{perturb_prop} shows that $\Lambda(x,y,z)=-xyz$ is a Lyapunov function as defined in \cite{Pemantle} for the vector field $P$ (since it is decreasing, as $xyz$ is increasing). Hence, by Proposition 2.18 of \cite{Pemantle} this process must converge to a stationary point of $P$, and the only such point is $(1/3, 1/3, 1/3)$. This proves Theorem \ref{perturb_thm}.
\end{proof}

\section{Knockout tournament case with $m=4$}
In this section, we will consider a new model which is a version of the rock-paper-scissors model of \cite{hj} but with $m=4$, so that each new node is connected to four existing nodes.  In our model, these four nodes then perform a knockout tournament, again following rock-paper-scissors rules, to decide the type of the new node. Specifically, the four nodes are paired off into two matchings, and the winner of each matching competes in the final.

Simulation results suggest that this case, despite that it may seem to have similar properties to the original rock-paper-scissors setup, converges. See Figure 2 for results from a $10^9$ step simulation; the proportions of each type settle around $1/3$ quite quickly.

\begin{figure}
  \includegraphics[scale=0.7]{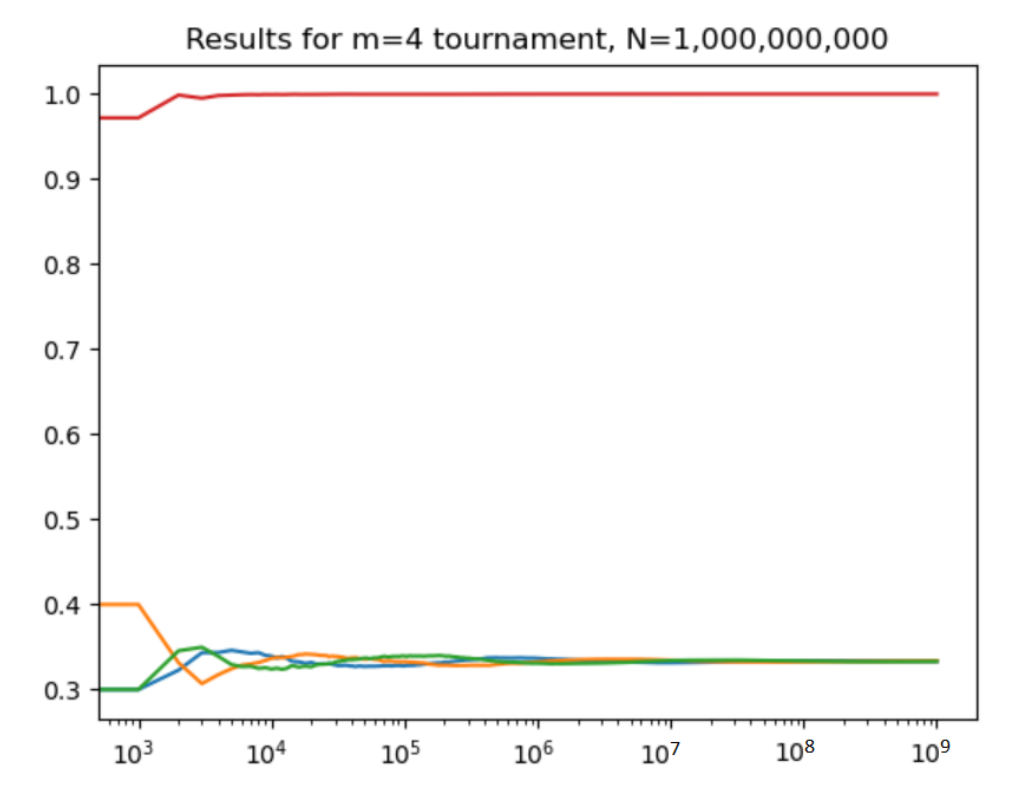}
  \caption{Results for a simulation of the tournament model with $m=4$. As before, here $G_0$ = a complete graph with three of each type. The red line shows the value of the product $27X_nY_nZ_n$. }
  \label{fig:2}
\end{figure}

We will prove the following theorem:

\begin{theorem}\label{tournament_thm}
Assume that each type is present in the initial graph $G_0$.  Then, for the $m=4$ tournament model, $(A_n, B_n, C_n)$ and $(X_n, Y_n, Z_n)$ converge almost surely to $(1/3,1/3,1/3)$, where $X_n, Y_n, Z_n$ denote the (normalised) sums of degrees of the nodes of types 1, 2 and 3 respectively in $G_n$.
\end{theorem}

There are many possible scenarios based on the initial matching. What happens in each case is detailed below. These will inform the formulas for the $p(\textbf{u})$.

\begin{itemize}
\item All four nodes are the same: the new node takes this type with probability one.
\item Two types of nodes are present: the new node takes the type of whichever would win a heads up contest, with probability one.
\item All three types of nodes are present: in this case, there is one type which is present twice, and the others are present once each. There are two possibilities. First, the two duplicates may be matched up in round one, and then the duplicated type will win (as it will face the type it beats in the final). This happens with probability $1/3$. Otherwise, the type that beats the duplicated type will win (as it will face the type it beats in both rounds). This happens with probability $2/3$. The final type, which loses to the duplicated type, cannot win.
\end{itemize}
From this understanding, we derive the following (we define type 1 to be "rock", type 2 to be "paper" and type 3 to be "scissors"):

$$p_{(4,0,0)} = p_{(3,0,1)} = p_{(2,0,2)} = p_{(1,0,3)} = (1,0,0),$$
$$p_{(0,4,0)} = p_{(1,3,0)} = p_{(2,2,0)} = p_{(3,1,0)} = (0,1,0),$$
$$p_{(0,0,4)} = p_{(0,1,3)} = p_{(0,2,2)} = p_{(0,3,1)} = (0,0,1),$$
$$p_{(2,1,1)} =  \left( \frac{1}{3},\frac{2}{3},0 \right), p_{(1,2,1)}=\left(0,\frac{1}{3},\frac{2}{3}\right), p_{(1,1,2)}=\left(\frac{2}{3},0,\frac{1}{3}\right)$$

The vector field (defined by (3.1) of \cite{amr}, as with the perturbation case) is given by the components
$$P_1(x,y,z) = \frac{x}{2}(-3x^2y+x^2z-3xy^2-2xyz+3xz^2-y^3-3y^2z+5yz^2+3z^3),$$
$$P_2(x,y,z) = \frac{y}{2}(-3y^2z+y^2x-3yz^2-2xyz+3yx^2-z^3-3z^2x+5zx^2+3x^3),$$
$$P_3(x,y,z) = \frac{z}{2}(-3z^2x+z^2y-3zx^2-2xyz+3zy^2-x^3-3x^2y+5xy^2+3y^3).$$

Our approach is now similar to that in the previous section: we will show that $\Lambda(x,y,z)=-xyz$ is a Lyapunov function for the vector field $P$, and thus deduce convergence of the underlying stochastic approximation process.

\begin{proposition}\label{tournament_prop}
The product $xyz$ is constant on the trajectories of $P$ only when either $(x,y,z) = (1/3,1/3,1/3)$ or at least one of $x, y, z$ is zero. Otherwise, it is increasing on said trajectories.
\end{proposition}

\begin{proof}
We have $\frac{d(xyz)}{dt} = xyP_3 + xzP_2 + yzP_1$. Substituting in the components above gives:
$$\frac{d(xyz)}{dt} = \frac{xyz}{2}(3x^2z+3xy^2+3yz^2-3x^2y-3y^2z-3xz^2$$
$$-6xyz+2x^3+2y^3+2z^3).$$

To find the zeroes of this function, we first consider the zeroes of $\frac{xyz}{2}$: they are precisely when one or more of $x,y,z$ is zero.

For all other cases, $\frac{xyz}{2}$ is non-zero, and we define $f = 3x^2z+3xy^2+3yz^2-3x^2y-3y^2z-3xz^2-6xyz+2x^3+2y^3+2z^3$, which has the same zeroes as $\frac{d(xyz)}{dt}$ away from the edges of the triangle. Since $x+y+z=1$, $z=1-x-y$. Substituting $z=1-x-y$ and expanding gives
$$f = 2 - 9x - 3y + 15x^2 + 6xy - 3y^2 - 6x^3 - 9x^2y + 9xy^2 + 6y^3.$$

The partial derivatives of $f$ are $\frac{\partial f}{\partial x} = 3(-3+10x+2y-6x^2-6xy+3y^2)$ and $\frac{\partial f}{\partial y} = 3(-1+2x-2y-3x^2+6xy+6y^2)$. Re-arranging the first, we obtain that $y = \frac{1}{3}(\sqrt{27x^2-36x+10}+3x-1)$. Substituting this into $\frac{\partial f}{\partial y}$ gives that, at stationary points, $(3x-1)(2\sqrt{27x^2-36x+10}+9x-7)=0$. This implies that either $x = 1/3$ or $x = -1/3$. Since here $x$ is always non-negative, the only relevant solution is $x = 1/3$. Symmetrical reasoning implies that $y = 1/3$ and $z = 1/3$ are satisfied at any zero of $f$. Hence, our only stationary point from $f$ is $(1/3,1/3,1/3)$.

Calculating second derivatives gives $\frac{\partial^2 f}{\partial x^2} = 6(5-6x-3y)$, $\frac{\partial^2 f}{\partial y^2} = 6(-1+3x+6y)$ and $\frac{\partial^2 f}{\partial x \partial y} = 6(1-3x+3y)$. At $(1/3,1/3,1/3)$, these are $12$, $12$ and $6$ respectively, and since $12$ is positive and $12 * 12 > 6^2$, we have that $(1/3,1/3,1/3)$ is a local minimum.

The only way this local minimum is not a global minimum is if the value of the function is negative somewhere on the edges of the triangle. But, $\frac{xyz}{2}$ is zero at all these points and so $\frac{d(xyz)}{dt}$ is zero. Hence, $(1/3,1/3,1/3)$ is a global minimum.

\end{proof}

We are now able to prove Theorem \ref{tournament_thm}, in a similar vein to in section \ref{perturb_sect}.
\begin{proof}[Proof of Theorem \ref{tournament_thm}.]
Proposition \ref{tournament_prop} shows that $\Lambda(x,y,z)=-xyz$ is a Lyapunov function for the vector field $P$. Hence, by Proposition 2.18 of \cite{Pemantle}, this process must converge almost surely to a stationary point of $P$.

It remains to check that the convergence must be to $(1/3,1/3,1/3)$.  Proposition \ref{tournament_prop} shows there are no other stationary points in the interior of $\Delta^2$, and straightforward analysis of $P$ on the boundary of $\Delta^2$ shows that the only other stationary points are the corners $(1,0,0)$, $(0,1,0)$ and $(0,0,1)$, each of which is a linearly unstable saddle point.

To show that $(1,0,0)$ is a limit with probability zero, assume that for some $\epsilon>0$ and for $n\geq n_0$ we have $Y_n,Z_n\leq \epsilon$, and consider the following coupling to a two type process.  Merge the rock and scissors types as ``red'', and consider the paper type as ``blue''.  Then, conditional on observing three red neighbours and one blue neighbour, for large $n$ the probability that all three red neighbours are in fact rock is at least $1-6\epsilon$ when $n\geq n_0$, and in this case the new vertex will be paper.  Similarly if there are two red neighbours the probability the new vertex is paper is at least $1-4\epsilon$, and if there is one red neighbour this probability is at least $1-2\epsilon$.  Hence, for $n\geq n_0$ the probability the new vertex is paper is at least as large as that in a two type process with, in the notation of \cite{amr}, $p_0=0,p_1=2\epsilon, p_2=4\epsilon, p_3=6\epsilon, p_4=1$.  For $\epsilon$ sufficiently small Theorem 1.4 of \cite{amr} shows that this two type process does not have positive probability of convergence to red domination as long as some blue vertices are present initially, and so convergence to $(1,0,0)$ cannot have positive probability in our model.  Analogous arguments apply to $(0,1,0)$ and $(0,0,1)$.
\end{proof}

\section*{Acknowledgement}
The authors acknowledge the support of the Undergraduate Research Internship programme in the School of Mathematics and Statistics at the University of Sheffield, funded from a bequest from Chris Cannings.

\end{document}